\documentclass[12pt]{amsart}
\usepackage{amsmath, amsthm, amssymb, amsopn, amsfonts, enumerate}
\usepackage{stmaryrd}
\newtheorem{thm}{Theorem}
\newtheorem{prop}[thm]{Proposition}
\newtheorem{lemma}[thm]{Lemma}
\newtheorem{conj}{Conjecture}
\newtheorem*{question}{Question}

\newcommand{\Mbar}{\overline{M}}

\newcommand{\PP}{\mathcal{P}}

\DeclareMathOperator{\Pic}{Pic}
\DeclareMathOperator{\PSL}{PSL}
\DeclareMathOperator{\SL}{SL}

\title{\textbf{A non-boundary nef divisor on $\Mbar_{0,12}$}}
\author{A. Pixton}
\date{December 2011}
\begin{document}

\maketitle

\begin{abstract}
We describe a nef divisor $D_\PP$ on $\Mbar_{0,12}$ that is not numerically equivalent to an effective sum of boundary divisors.
\end{abstract}

\setcounter{section}{-1}
\section{Introduction}
The moduli space of stable curves of genus $0$ with $n$ marked points, $\Mbar_{0,n}$, has a stratification where the codimension $k$ strata are the irreducible components of the locus of curves with at least $k$ nodes. Fulton, motivated by analogy with toric varieties, asked whether any effective cycle is numerically equivalent to an effective sum of these strata (see \cite{Keel-McKernan}, Question 1.1). In the case of divisors, this is known to be false; Keel and Vermeire \cite{Vermeire} constructed effective divisors on $\Mbar_{0,6}$ that are not effective sums of boundary divisors (the codimension 1 strata).

For curves, however, the question is still open. The $1$-dimensional strata are called {\it F-curves} and the conjecture in this case (known as the F-conjecture) is that the cone of effective curves is generated by the F-curves. It is often stated in the following equivalent form:
\begin{conj}[The F-conjecture]
A divisor on $\Mbar_{0,n}$ is nef if and only if it has nonnegative intersection with every F-curve.
\end{conj}
Keel and McKernan \cite{Keel-McKernan} proved this conjecture for $n\le 7$, and for $n \ge 8$ the conjecture is still open. We usually say that a divisor that intersects every F-curve nonnegatively is {\it F-nef}, so the F-conjecture states that a divisor on $\Mbar_{0,n}$ is nef if and only if it is F-nef.

The following related conjecture will be the central subject of this paper. It is sometimes called ``Fulton's Conjecture,'' though it seems to have first been stated by Gibney, Keel, and Morrison \cite{Gibney-Keel-Morrison}.
\begin{conj}[\cite{Gibney-Keel-Morrison}, Question (0.13)]\label{conj_not_fulton}
Every F-nef divisor on $\Mbar_{0,n}$ is numerically equivalent to an effective sum of boundary divisors.
\end{conj}
This conjecture has also been proven for $n\le 7$ \cite{Larsen}, and it has long been understood that it would imply the F-conjecture by a straightforward inductive argument, if true.

Gibney showed that this conjecture can be weakened slightly and still imply the F-conjecture:
\begin{conj}[\cite{Gibney}, Conjecture 1]\label{conj_gibney}
Every F-nef divisor on $\Mbar_{0,n}$ is of the form $c K_{\Mbar_{0,n}} + E$ where $c\ge 0$ and $E$ is an effective sum of boundary classes.
\end{conj}

However, we provide a counterexample to Conjectures~\ref{conj_not_fulton} and \ref{conj_gibney} when $n = 12$.

\begin{prop}\label{counterexample}
The divisor $D_\PP$ on $\Mbar_{0,12}$ is F-nef but is not numerically equivalent to a nonnegative linear combination of boundary divisors and the canonical divisor $K_{\Mbar_{0,12}}$.
\end{prop}

We currently have no geometric explanation for this counterexample, but it is related to a highly symmetric combinatorial arrangement of subsets, the $(11,5,2)$ biplane (see Section~\ref{biplane}).

We do not know whether $n=12$ is the minimal $n$ for which Conjecture~\ref{conj_not_fulton} is false. We do give a simple argument (Lemma~\ref{pullback}) allowing us to shift the counterexample to any larger $n$, though, so there are only four values of $n$ for which the status of Conjecture~\ref{conj_not_fulton} is still unknown.
\begin{question}
Does Conjecture~\ref{conj_not_fulton} hold for $\Mbar_{0,n}$ with $8 \le n \le 11$?
\end{question}

We begin by reviewing notation in Section~\ref{notation}. In Section~\ref{construction}, we describe the counterexample divisor $D_\PP$ appearing in Proposition~\ref{counterexample} and show that its pullbacks under forgetful maps also provide counterexamples. In Section~\ref{computations}, we list a few additional properties of this divisor calculated by computer; in particular, it is basepoint-free, and hence nef.

\subsection{Acknowledgements}
I am grateful to R. Pandharipande for many helpful discussions about $\Mbar_{0,n}$. I would also like to thank A. Castravet and J. Tevelev for their computation described in Section~\ref{computations}, and J. Tevelev for discussions that were helpful for proving Lemma~\ref{pullback}.

My research was supported by a NDSEG graduate fellowship.

\section{Notation}\label{notation}

\subsection{The $(11, 5, 2)$ biplane}\label{biplane}
The divisor we construct is based on the combinatorial data of the $(11, 5, 2)$ biplane, a collection $\PP$ of $11$ five-element subsets of $[11] := \{1,2,\ldots,11\}$ with the properties:
\begin{enumerate}
\item Any two elements of $[11]$ are in precisely two elements of $\PP$.
\item Any two elements of $\PP$ have intersection of size two. 
\end{enumerate}
It turns out that $\PP$ is unique up to renaming the elements of $[11]$; one way of constructing such a biplane is to take the set $\{1,3,4,5,9\}$ (the nonzero quadratic residues mod $11$) together with its translates (taken mod $11$).

The $(11,5,2)$ biplane has symmetry group of order $660$, isomorphic to $\PSL_2(11)$. The divisor $D_\PP$ that we construct shares this symmetry group.

\subsection{Divisors on $\Mbar_{0,n}$}

The simplest divisors on $\Mbar_{0,n}$ are the boundary divisors $\Delta_{S,T} = \Delta_S = \Delta_T$, where $\{S,T\}$ is a partition of $\{1,\ldots,n\}$ with $|S|,|T|\ge 2$.

It is also convenient to define $\Delta_{\{i\}} = -\psi_i$, the negative of the cotangent line at the $i$th marked point.

These two types of divisors suffice to generate $\Pic(\Mbar_{0,n})$ (in fact, the boundary divisors alone do), and the space of relations between these generators has basis consisting of the relations
\begin{equation}\label{relations}
\sum_{i\in S, j\notin S}\Delta_S = 0
\end{equation}
for $1\le i < j\le n$.

\subsection{F-curves}

The dual graph to an F-curve has exactly one vertex of degree $4$. The class of the F-curve just depends on the partition of the markings into $4$ subsets determined by this vertex. If $\{A_1,A_2,A_3,A_4\}$ is a partition of $\{1,\ldots,n\}$ into four nonempty subsets, we write $C_{A_1A_2A_3A_4}$ for the corresponding F-curve class.

The intersection number of any F-curve with any divisor $\Delta_{S,T}$ has a simple formula:
\begin{equation}\label{pairing}
\Delta_{S,T}\cdot C_{A_1A_2A_3A_4} = \begin{cases}
1 \text{ if }S = A_i\cup A_j\text{ for some $i\ne j$} \\
-1 \text{ if }S\text{ or }T = A_i\text{ for some $i$} \\
0 \text{ else}. \end{cases}
\end{equation}

\section{The construction}\label{construction}
We begin by describing a curve class $C_\PP$ that directly corresponds to the biplane configuration $\PP$. It is uniquely determined by its pairings with the boundary divisors $\Delta_{S,T}$, since they generate the Picard group:
\[
C_\PP\cdot \Delta_{S,T} = \begin{cases}
1 \text{ if } S\text{ or }T \in\PP \\
0 \text{ else}. \end{cases}
\]
Of course, for such a divisor to exist, these pairings have to respect the relations between the divisors $\Delta_{S,T}$. The fact that they do turns out to be equivalent to every pair of elements of $[11]$ being contained in the same number of sets in the biplane $\PP$. This is actually the only biplane property that we will directly use. Of course, there are many similarly balanced configurations of subsets that we could consider using instead of $\PP$ to construct different interesting curve classes.

Since $C_\PP$ intersects every boundary divisor nonnegatively, any divisor that intersects it negatively cannot be an effective sum of boundary. Thus we view $C_\PP$ as a witness to the non-boundary nature of certain divisors.

We can also compute that
\[
C_\PP\cdot \Delta_{\{i\}} = \begin{cases}
-2 \text{ if } i = 12 \\
-3 \text{ else}. \end{cases}
\]

We now describe the divisor itself. Unlike with $C_\PP$, the definition of $D_\PP$ does not seem to have an obvious generalization to other combinatorial configurations of subsets.

The notation $E_S = \Delta_{S \cup \{n\}}$ will be convenient here; these are the exceptional divisors in the Kapranov model of $\Mbar_{0,n}$ with respect to the $n$th marked point. Then we define the divisor $D_\PP \in \Pic(\Mbar_{0,12})$ by
\[
D_\PP = -5 E_{\emptyset} - 4\sum E_i - 3\sum E_{ij} - 2\sum E_{ijk} - \sum E_{ijkl} - \sum E_S,
\]
where the first four sums are over all subsets of $[11]$ of sizes $1,2,3,4$ and the last sum runs over the subsets of $[11]$ of size $5$ or $6$ that are either equal to or disjoint from one of the eleven five-element subsets in the chosen biplane $\PP$.

We can now check that we have a counterexample to Conjectures~\ref{conj_not_fulton} and \ref{conj_gibney}.
\begin{proof}[Proof of Proposition~\ref{counterexample}]
This is just a matter of checking four things:
\begin{enumerate}[(a)]
\item $D_\PP \cdot C_{A_1A_2A_3A_4} \ge 0$ for any F-curve class $C_{A_1A_2A_3A_4}$
\item $\Delta_{S,T} \cdot C_\PP \ge 0$ for $|S|,|T| \ge 2$
\item $K_{\Mbar_{0,12}} \cdot C_\PP \ge 0$
\item $D_\PP \cdot C_\PP < 0$.
\end{enumerate}

Of these, (b), (c), and (d) all follow immediately from the definition of $C_\PP$, combined with the identity
\[
K_{\Mbar_{0,12}} = -\sum_{i=1}^{12}\Delta_{\{i\}} -2\sum_{|S|,|T| \ge 2}\Delta_{S,T}
\]
in the case of (c).

This leaves (a), which is just a matter of computing the intersection of $D_\PP$ with each of the finitely many F-curve classes on $\Mbar_{0,12}$. There are $611501$ such classes, so this check can be easily completed with a computer using \eqref{pairing}. Alternatively, we can write $D_\PP = D_0 - D'_\PP$, where
\[
D_0 = -5 E_{\emptyset} - 4\sum E_i - 3\sum E_{ij} - 2\sum E_{ijk} - \sum E_{ijkl}
\]
and 
\[
D'_\PP = \sum_{S\in\PP}\left(\Delta_S + \sum_{i\notin S}\Delta_{S\cup\{i\}}\right),
\]

The divisor $D_0$ here is very special; it is fully $S_{12}$-symmetric and can either be interpreted as the pullback of the distinguished polarization of the (symmetrically linearized) GIT quotient $(\mathbb{P}^1)^{12}\sslash \SL_2$ (see \cite{Alexeev-Swinarski}) or as a conformal block divisor (\cite{Fakhruddin}, Theorem 4.5). It is nef, and its degree on an F-curve $C_{A_1A_2A_3A_4}$ has a simple formula:
\[
D_0\cdot C_{A_1A_2A_3A_4} = \begin{cases}
0&\text{ if } \max(|A_i|) \ge 6 \\
\min(\min(|A_i|), 6 - \max(|A_i|))&\text{ else}. \end{cases}
\]

Using this formula along with \eqref{pairing}, it is straightforward to check that $D_0\cdot C_{A_1A_2A_3A_4} \ge D'_\PP\cdot C_{A_1A_2A_3A_4}$ for any F-curve $C_{A_1A_2A_3A_4}$. For example, if $D_0\cdot C_{A_1A_2A_3A_4} = 0$ then $|A_1| \ge 6$ (without loss of generality), and then the only possibilty of $D'_\PP\cdot C_{A_1A_2A_3A_4}$ being positive is if $A_2\cup A_3 \in \PP$. But then $A_2\cup A_4$ and $A_3\cup A_4$ cannot be elements of $\PP$ (since the union of any two sets in $\PP$ has cardinality $8 > 12 - 6$), and $A_2\cup A_3\cup A_4$ is an element of $\PP$ with a single point added, so $D'_\PP\cdot C_{A_1A_2A_3A_4} = 1 - 1 = 0$. The other cases are similar.

\end{proof}

Although we expect that both Conjecture~\ref{conj_not_fulton} and Conjecture~\ref{conj_gibney} remain false for all $n\ge 12$, we are only able to prove this for Conjecture~\ref{conj_not_fulton} at the moment.

\begin{lemma}\label{pullback}
Let $n\ge 3$ and let $\pi:\Mbar_{0,n+1}\to\Mbar_{0,n}$ be the map given by forgetting the final marking. If $D$ is an F-nef divisor on $\Mbar_{0,n}$ that is not numerically equivalent to an effective boundary divisor, then $\pi^* D$ also is F-nef and not numerically equivalent to an effective boundary divisor.
\end{lemma}
\begin{proof}
First, $\pi^* D$ being F-nef follows immediately from the fact that for any F-curve class $C$, $\pi_* C$ is either $0$ or another F-curve class.

Now suppose for contradiction that $\pi^* D$ is numerically equivalent to
\[
\sum a_{S,T}\Delta_{S,T}
\]
with $a_{S,T}\ge 0$, where the sum runs over all partitions of $\{1,\ldots,n+1\}$ into two sets $S$ and $T$ of cardinality at least $2$.

The F-curve $C_{A,B,C,\{n+1\}}$ is contracted by $\pi$, so we have
\begin{equation}\label{contraction}
\begin{split}
0 &= \pi^* D\cdot C_{A,B,C,\{n+1\}} \\
&= (a_{A\cup\{n+1\},B\cup C}-a_{A,B\cup C\cup\{n+1\}})\\ &\quad+ (a_{B\cup\{n+1\},A\cup C}-a_{B,A\cup C\cup\{n+1\}})\\ &\quad- (a_{A\cup B\cup\{n+1\},C}-a_{A\cup B,C\cup\{n+1\}}).
\end{split}
\end{equation}
Let $f(S) = a_{S\cup\{n+1\},S^c} - a_{S,S^c\cup\{n+1\}}$ for $\emptyset\ne S\subset\{1,\ldots,n\}$, where $S^c$ is the complement of $S$ inside $\{1,\ldots,n\}$. Then \eqref{contraction} becomes
\[
f(A\cup B) = f(A) + f(B),
\]
where $A$ and $B$ are any two disjoint nonempty subsets of $\{1,\ldots,n\}$ such that $A\cup B$ is a proper subset of $\{1,\ldots,n\}$.

We also have $f(A^c) = -f(A)$ by the definition of $f$. Thus 
\[
f(\{1\}) + \cdots + f(\{n\}) = 0.
\]
However, for each $k$ we have $f(\{k\}) = a_{\{k,n+1\},\{k\}^c} \ge 0$. Thus $f(\{k\}) = 0$ for all $k$, and hence $f(S) = 0$ for all $S$. In other words, $a_{S\cup\{n+1\},S^c} = a_{S,S^c\cup\{n+1\}}$.

Note now that $\pi^*\Delta_{S,S^c} = \Delta_{S\cup\{n+1\},S^c} + \Delta_{S,S^c\cup\{n+1\}}$. Then $\pi^* D$ is numerically equivalent to $\pi^*B$ for 
\[
B = \sum_{\{S,S_c\}}a_{S\cup\{n+1\},S^c}\Delta_{S,S^c}.
\]
But $\pi^*:\Pic(\Mbar_{0,n})\to\Pic(\Mbar_{0,n+1})$ in injective (because $\pi$ has a section), so this implies that $D$ is numerically equivalent to $B$, which is effective boundary. This is a contradiction.
\end{proof}

\section{Computations}\label{computations}

We have checked a few basic properties of the divisor $D_\PP$ by computer.
First, $D_\PP$ generates an extremal ray in the cone of F-nef divisors on $\Mbar_{0,12}$. We do not know whether there are any extremal rays on $\Mbar_{0,12}$ other than the $S_{12}$-orbit containing $D_\PP$ that contradict Conjecture~\ref{conj_not_fulton}. 

Also, $D_\PP$ is effective, with a $66$-dimensional space of sections. By straightforward computation of the intersection of these sections using the Kapranov model for $\Mbar_{0,12}$ as an iterated blow-up of $\mathbb{P}^9$, we were able to check that this linear system is basepoint-free. Thus $D_\PP$ is nef, which can be interpreted as very weak evidence for the F-conjecture.

Since $D_\PP$ is nef and extremal F-nef, it is also an extremal nef divisor. Unsurprisingly, it is not extremal in the cone of effective divisors. Castravet and Tevelev have computed that $D_\PP$ is in fact an effective sum of boundary divisors and a pullback from $\Mbar_{0,8}$ of the unique (up to symmetries) ``hypertree'' divisor there, so it is consistent with their conjecture that the cone of effective divisors is generated by the boundary divisors and the hypertree divisors (defined in \cite{Castravet-Tevelev}). 

Finally, we can give a somewhat more geometric description of $D_\PP$, again in terms of the Kapranov morphism $\kappa_{12}:\Mbar_{0,12}\to\mathbb{P}^9$. If $Q$ is a generic quintic hypersurface in $\mathbb{P}^9$ containing the eleven $\mathbb{P}^4$s spanned by the five-tuples in $\PP$ along with the eleven complementary $\mathbb{P}^5$s, then $D_\PP$ is the class of the proper transform of $Q$ with respect to $\kappa_{12}$.

\bibliographystyle{amsplain}
\bibliography{Fcurve}

\providecommand{\bysame}{\leavevmode\hbox to3em{\hrulefill}\thinspace}
\providecommand{\MR}{\relax\ifhmode\unskip\space\fi MR }
\providecommand{\MRhref}[2]{%
  \href{http://www.ams.org/mathscinet-getitem?mr=#1}{#2}
}
\providecommand{\href}[2]{#2}
\begin{thebibliography}{1}

\bibitem{Alexeev-Swinarski}
V.~Alexeev and D.~Swinarski, \emph{Nef divisors on $\overline{M}_{0,n}$ from
  {G}{I}{T}}, arXiv:0812.0778v2.

\bibitem{Castravet-Tevelev}
A.~Castravet and J.~Tevelev, \emph{Hypertrees, projections, and moduli of
  stable rational curves}, arXiv:1004.2553v2.

\bibitem{Fakhruddin}
N.~Fakhruddin, \emph{Chern classes of conformal blocks}, arXiv:0904.2918v4.

\bibitem{Gibney}
A.~Gibney, \emph{Numerical criteria for divisors on {$\overline M_g$} to be
  ample}, Compos. Math. \textbf{145} (2009), no.~5, 1227--1248.

\bibitem{Gibney-Keel-Morrison}
A.~Gibney, S.~Keel, and I.~Morrison, \emph{Towards the ample cone of
  {$\overline M_{g,n}$}}, J. Amer. Math. Soc. \textbf{15} (2002), no.~2,
  273--294.

\bibitem{Keel-McKernan}
S.~Keel and J.~McKernan, \emph{Contractible extremal rays on
  $\overline{M}_{0,n}$}, arXiv:alg-geom/9607009v1.

\bibitem{Larsen}
P.~Larsen, \emph{Fulton's conjecture for $\overline{M}_{0,7}$},
  arXiv:0912.3104v2.

\bibitem{Vermeire}
P.~Vermeire, \emph{A counterexample to {F}ulton's conjecture on {$\overline
  M_{0,n}$}}, J. Algebra \textbf{248} (2002), no.~2, 780--784.

\end{thebibliography}

\vspace{+8 pt}
\noindent
Department of Mathematics\\
Princeton University\\
apixton@math.princeton.edu

\end{document}